\documentclass[notitlepage,11pt,reqno]{amsart}
\usepackage[foot]{amsaddr}
\usepackage[numbers,sort]{natbib}

\usepackage{amssymb,nicefrac,bm,upgreek,mathtools,verbatim}
\usepackage[final]{hyperref}
\usepackage[mathscr]{eucal}
\usepackage{dsfont}

\usepackage{color}
\definecolor{dmagenta}{rgb}{.4,.1,.5}
\definecolor{dblue}{rgb}{.0,.0,.5}
\definecolor{mblue}{rgb}{.0,.0,.8}
\definecolor{ddblue}{rgb}{.0,.0,.4}
\definecolor{dred}{rgb}{.6,.0,.0}
\definecolor{dgreen}{rgb}{.0,.5,.0}
\definecolor{Eeom}{rgb}{.0,.0,.5}

\usepackage[normalem]{ulem}
\usepackage[margin=1in]{geometry}
\allowdisplaybreaks
\newtheorem{theorem}{Theorem}[section]

\newtheorem{corollary}{Corollary}[section]

\theoremstyle{definition}
\newtheorem{definition}{Definition}[section]
\newtheorem{assumption}{Assumption}[section]

\newtheorem{example}{Example}[section]
\newtheorem{remark}{Remark}[section]

\numberwithin{equation}{section}

\hypersetup{
  colorlinks=true,
  citecolor=dblue,
  linkcolor=dblue,
  frenchlinks=false,
  pdfborder={0 0 0},
  naturalnames=false,
  hypertexnames=false,
  breaklinks}
\usepackage[capitalize,nameinlink]{cleveref}

\crefname{section}{Section}{Sections}
\crefname{subsection}{Subsection}{Subsections}
\crefname{condition}{Condition}{Conditions}
\crefname{hypothesis}{Hypothesis}{Conditions}
\crefname{assumption}{Assumption}{Assumptions}
\crefname{lemma}{Lemma}{Lemmas}
\crefname{claim}{Claim}{Claims}

\Crefname{figure}{Figure}{Figures}

\crefformat{equation}{\textup{#2(#1)#3}}
\crefrangeformat{equation}{\textup{#3(#1)#4--#5(#2)#6}}
\crefmultiformat{equation}{\textup{#2(#1)#3}}{ and \textup{#2(#1)#3}}
{, \textup{#2(#1)#3}}{, and \textup{#2(#1)#3}}
\crefrangemultiformat{equation}{\textup{#3(#1)#4--#5(#2)#6}}%
{ and \textup{#3(#1)#4--#5(#2)#6}}{, \textup{#3(#1)#4--#5(#2)#6}}%
{, and \textup{#3(#1)#4--#5(#2)#6}}

\Crefformat{equation}{#2Equation~\textup{(#1)}#3}
\Crefrangeformat{equation}{Equations~\textup{#3(#1)#4--#5(#2)#6}}
\Crefmultiformat{equation}{Equations~\textup{#2(#1)#3}}{ and \textup{#2(#1)#3}}
{, \textup{#2(#1)#3}}{, and \textup{#2(#1)#3}}
\Crefrangemultiformat{equation}{Equations~\textup{#3(#1)#4--#5(#2)#6}}%
{ and \textup{#3(#1)#4--#5(#2)#6}}{, \textup{#3(#1)#4--#5(#2)#6}}%
{, and \textup{#3(#1)#4--#5(#2)#6}}

\crefdefaultlabelformat{#2\textup{#1}#3}

\makeatletter
\DeclareRobustCommand\widecheck[1]{{\mathpalette\@widecheck{#1}}}
\def\@widecheck#1#2{%
    \setbox\z@\hbox{\m@th$#1#2$}%
    \setbox\tw@\hbox{\m@th$#1%
       \widehat{%
          \vrule\@width\z@\@height\ht\z@
          \vrule\@height\z@\@width\wd\z@}$}%
    \dp\tw@-\ht\z@
    \@tempdima\ht\z@ \advance\@tempdima2\ht\tw@ \divide\@tempdima\thr@@
    \setbox\tw@\hbox{%
       \raise\@tempdima\hbox{\scalebox{1}[-1]{\lower\@tempdima\box
\tw@}}}%
    {\ooalign{\box\tw@ \cr \box\z@}}}

\def\subsection{\@startsection{subsection}{0}%
\z@{\linespacing\@plus\linespacing}{\linespacing}%
{\bf}}

\makeatother

\DeclareMathOperator{\Prob}{\mathbb{P}} %Probability
\newcommand{\D}{\mathrm{d}}          %differential
\newcommand{\RR}{\mathbb{R}}         %Real numbers
\newcommand{\Rd}{{\mathbb{R}^d}}       %R^d
\newcommand{\NN}{\mathbb{N}}         %Natural numbers
\newcommand{\Ind}{\mathds{1}}            %indicator function
\newcommand{\cC}{\mathcal{C}}     % Classes of continuous functions
\newcommand{\cD}{\mathcal{D}}     % Domain of \Rd
\newcommand{\cH}{\mathcal{H}}
\newcommand{\cB}{\mathcal{B}}

\newcommand{\sH}{\mathscr{H}}

\newcommand{\cL}{\mathcal{L}}
\newcommand{\abs}[1]{\lvert#1\rvert}
\newcommand{\norm}[1]{\lVert#1\rVert}

\providecommand{\pro}[1]{(#1_t)_{t \geq 0}}

\providecommand{\semi}[1]{\{#1_t: t \geq 0\}}
\providecommand{\seq}[1]{(#1_n)_{n\in \mathbb{N}}}

\DeclareMathOperator{\Dom}{Dom}

\DeclareMathOperator{\diam}{diam}

\newcommand{\ex}{\mathbb{E}}

\DeclareMathOperator*{\esssup}{ess\,sup}

\DeclareMathOperator{\Psidel}{\Psi(-\Delta)}

\begin{document}

\title[Maximum principles]%
{\sc \textbf{Maximum principles for time-fractional Cauchy problems with spatially non-local components}}

\author{Anup Biswas and J\'{o}zsef L\H{o}rinczi}

\address{Anup Biswas \\
Department of Mathematics, Indian Institute of Science Education and Research, Dr. Homi Bhabha Road,
Pune 411008, India, anup@iiserpune.ac.in}

\address{J\'ozsef L\H{o}rinczi \\
Department of Mathematical Sciences, Loughborough University, Loughborough LE11 3TU, United Kingdom,
J.Lorinczi@lboro.ac.uk}

\date{}

%%%%%%%%%%%%%%%%%%%%%%%%%%%%%%%%%%%%%%%%%%%%%%%%%%%%%%%%%%%%%%%%%%%%%%%%%%%%%%%%

\begin{abstract}
We show a strong maximum principle and an Alexandrov-Bakelman-Pucci estimate for the weak solutions of a
Cauchy problem featuring Caputo time-derivatives and non-local operators in space variables given in terms
of Bernstein functions of the Laplacian. To achieve this, first we propose a suitable meaning of a weak
solution, show their existence and uniqueness, and establish a probabilistic representation in terms of
time-changed Brownian motion. As an application, we also discuss an inverse source problem.
\end{abstract}

\maketitle

\section{\bf Introduction}
Evolution equations featuring fractional time-derivatives currently receive much attention in both pure and applied
mathematics due to, on the one hand, new qualitative properties not encountered in the realm of PDE and, on the other
hand, for their novel modelling capabilities in science \cite{P99,MBSB,BM01,MNV,OB09,K15,BLM,OOT,BDST,C17}. There are
a number of inequivalent concepts of fractional derivatives in use, which are presently the object of a wide-ranging
study. One such concept is the Caputo-derivative, which for a given number $\alpha \in (0,1)$ and a suitable function
$f$ is defined by
$$
\frac{\D^\alpha f(t)}{\D t^\alpha}=\frac{1}{\Gamma(1-\alpha)}\int_0^t \frac{1}{(t-s)^{\alpha}}\frac{\D f(s)}{\D s}\,\D{s},
$$
where $\Gamma$ is the usual Gamma-function. Equations with Caputo-derivatives occur, for instance, in the context of
anomalous transport theory, arising as scaling limits of continuous-time random walk (CTRW) models. Fractional diffusion
models described by the equation
$$
\partial_t^\alpha p(t,x) = D(-\Delta)^{\nu/2} p(t,x),
$$
with Caputo time-derivatives of the order $\alpha \in (0,1)$ and fractional Laplacians in space variables of index $\nu
\in (0,2)$, and with diffusion constant $D > 0$, have been much studied in the literature; for a discussion see \cite{MS}
and the references therein. This equation is the continuum limit of a CTRW model in which the random walker makes jumps
$z \in \Rd$ whose probability distribution has a tail proportional to $|z|^{-\nu/2}$, separated by random waiting times
$T$ with tail distributions proportional to $T^{-\alpha}$, leading to a non-Gaussian behaviour of the limit process. This,
in particular, captures more realistically the empirically observed anomalous spread of contaminants in groundwater flow
through a porous soil, see e.g. \cite{G10}, which is just one of the multiple applications of such equations.

In the present paper our goal is to consider a whole class of integro-differential equations of the above type, in which
we maintain Caputo time-derivative but allow many other choices of non-local operators in the space variable. Let $\cD
\subset \Rd$ be a bounded domain, and $F: (0, \infty)\times \cD \to \RR$ and $V, \varphi_0: \Rd \to \RR$ be given functions,
subject to conditions to be specified below. Also, let $\Psi$ be a so called Bernstein function (see Section 2 for the
details), which we will use to define non-local (pseudo-differential) operators of the form $\Psidel$, where $\Delta$ is
the usual Laplacian. A specific choice is not only the usual or the fractional Laplacian above, but many others of interest
such as their sum (describing jump-diffusion), or relativistic, geometric etc type of stable operators covering applications
in relativistic quantum theory, laser physics etc (for a discussion see \cite{KL16}). With these ingredients, in this paper
our interest is to establish maximum principles for the solutions of integro-differential equations of the form
\begin{equation}\label{E1.1}
   \left\{\begin{aligned}
     \partial^\alpha_t \varphi + \Psidel\varphi + V \varphi \; &= \; F(t, x) \quad \text{in} \;\, (0, \infty)\times \cD  \\
     \varphi(t, x) \; &= \; 0 \quad\qquad\;\, \text{in} \;\, (0, \infty)\times \cD^c  \\
     \varphi(0, x) \; & = \; \varphi_0(x).
   \end{aligned}\right.
\end{equation}

Maximum principles and related Alexandrov-Bakelman-Pucci (ABP) estimates are results of fundamental relevance in the analysis
of PDE or, more recently, integro-differential equations. In \cite{BL17b} we have recently obtained so-called refined and
weak maximum principles, anti-maximum principles, ABP estimates, Liouville-type theorems and related results for elliptic
non-local equations of the type as the space-dependent part in \eqref{E1.1}, as well as a parabolic ABP estimate for the
sub-solutions for the case when time dependence entered via usual derivatives corresponding to $\alpha = 1$ above. (We also
refer to the introduction of this paper for a review of the state of the art in the PDE literature.) In particular, the latter
has the form
$$
\sup_{[0, T)\times \cD} \varphi^+
\leq
\left(\sup_{[0, T)\times\cD^c} \varphi^+ \vee \sup_{\{T\}\times\cD} \varphi^+\right)+
C \left(1 + \frac{1}{\left(\Psi([\diam \cD]^{-2})\right)^\frac{2}{p'}}\right) \norm{F}_{p, Q_T},
$$
for (the positive part of) a bounded weak sub-solution $\varphi$ of \eqref{E1.1} with $\alpha = 1$, where $\cD$ is assumed to
have a regularity property, $\Psi$ satisfies a one-sided weak scaling condition with parameter $\underline{\mu}$ (see stated
precisely in Assumption \ref{WLSC} below), $F\in L^p([0, T) \times \cD)$, $C = C(p, d, \Psi)$ is a constant, $p, p'$ are
H\"older-conjugate exponents, and where $p > \frac{d}{2{\underline\mu}} + 1$.

For time-fractional equations maximum principles have been considered in \cite{L09,LRY,L17,LY}, however, only for the
cases when instead of a non-local operator $\Psidel$ the Laplacian or a second order elliptic differential operator in
divergence form is used. Further developing our approach proposed in \cite{BL17b} to time-fractional equations, in this
paper we consider also non-local spatial dependence, going well beyond the results established by these authors. A counterpart
of the ABP estimate for the time-fractional case leads us to
$$
p > \frac{d}{2{\underline\mu}} + \frac{1}{\alpha},
$$
which gives then an explanation of the bound on $p$ above, with a neat separation of space-time contributions.

The remainder of this paper is organized as follows. In Section 2 we briefly discuss Bernstein functions and subordinate
Brownian motion, which are our main tools in describing the spatial part in the equations. In our main Section 3 first we
show existence and uniqueness of weak solutions of the Cauchy problem (Theorem \ref{exuni}), and then derive a probabilistic
representation of the solution in terms of expectations over time-changed Brownian motion (Theorem \ref{T2.2}). Next we
prove a strong maximum principle (Theorem \ref{T2.3}) for any bounded domain with regular boundary. In Theorem \ref{stab}
we present a result on the stability of solutions under varying the data. Then we derive an Alexandrov-Bakelman-Pucci type
estimate for the solutions under a one-sided weak scaling property on the Bernstein functions used (Theorem \ref{T2.4}).
Finally, we present an application to the inverse source problem (Theorem \ref{inver}), which has also a practical relevance.

\section{\textbf{Bernstein functions of the Laplacian and subordinate Brownian motions}}
Consider the set of non-negative, completely monotone functions
$$
\mathcal B = \left\{f \in \cC^\infty((0,\infty)): \, f \geq 0 \;\; \mbox{and} \:\; (-1)^n\frac{\D^n f}{\D x^n} \leq 0,
\; \mbox{for all $n \in \mathbb N$}\right\}
$$
and its subset
$$
{\mathcal B}_0 = \left\{f \in \mathcal B: \, \lim_{u\downarrow 0} f(u) = 0 \right\}.
$$
An element of $\mathcal B$ is called a Bernstein function, in particular, they are non-negative, increasing, and concave.
Let $\mathcal M$ be the set of Borel measures $\mathfrak m$ on $\RR \setminus \{0\}$ with the property that
$$
\mathfrak m((-\infty,0)) = 0 \quad \mbox{and} \quad \int_{\RR\setminus\{0\}} (y \wedge 1) \mathfrak m(dy) < \infty.
$$
It is known that Bernstein functions $\Psi \in {\mathcal B}_0$ can be represented in the form
$$
\Psi(u) = bu + \int_{(0,\infty)} (1 - e^{-yu}) \mathfrak m(\D{y})
$$
with $b \geq 0$, and the map $[0,\infty) \times \mathcal M \ni (b,\mathfrak m) \mapsto \Psi \in {\mathcal B}_0$ is
bijective.

\begin{example}\label{Eg2.1}
Some key examples of $\Psi$ include:
\begin{itemize}
\item[(i)]
$\Psi(u)=u^{\nu/2}, \, \nu\in(0, 2]$
\vspace{0.1cm}
\item[(ii)]
$\Psi(u)=(u+m^{2/\nu})^{\nu/2}-m$, $\nu\in (0, 2)$, $m> 0$
\vspace{0.1cm}
\item[(iii)]
$\Psi(u)=u^{\nu/2} + u^{\tilde \nu/2}, \, \nu, \tilde\nu \in(0, 2]$
\vspace{0.1cm}
\item[(iv)]
$\Psi(u)=u^{\nu/2}(\log(1+u))^{-\tilde\nu/2}$, $\nu \in (0,2]$, $\tilde\nu \in [0,\nu)$
\vspace{0.1cm}
\item[(v)]
$\Psi(u)=u^{\nu/2}(\log(1+u))^{\tilde\nu/2}$, $\nu \in (0,2)$, $\tilde\nu \in (0, 2-\nu)$.
\end{itemize}
\end{example}

Bernstein functions give the Laplace exponents of subordinators. Recall that an $\RR^+$-valued L\'evy process $\pro S$
on a probability space $(\Omega_S, {\mathcal F}_S, \mathbb P_S)$ is called a subordinator whenever $\mathbb P_S (S_s
\leq S_t) = 1$ for $s \leq t$. There is a bijection between the set of subordinators on a given probability space and
Bernstein functions in ${\mathcal B}_0$, and the relationship
\begin{equation*}
\label{lapla}
\ex_{\mathbb P_S} [e^{-uS^\Psi_t}] = e^{-t\Psi(u)}, \quad u, t \geq 0,
\end{equation*}
holds, where $\Psi \in {\mathcal B}_0$, and where we have written $\pro {S^\Psi}$ for the unique subordinator associated
with Bernstein function $\Psi$. Corresponding to the examples of Bernstein functions above, the related processes are (i)
$\nu/2$-stable subordinator, (ii) relativistic $\nu/2$-stable subordinator, (iii) mixtures of independent subordinators
of different indices etc.

Let $\pro B$ be $\Rd$-valued Brownian motion on Wiener space $(\Omega_W,{\mathcal F}_W, \mathbb P_W)$, running twice at
its normal speed so that it has variance $\ex_{\Prob_W} [B_t^2] = 2t$, $t\geq 0$. Also, let $\pro {S^\Psi}$ be an
independent subordinator. The random process
$$
\Omega_W \times \Omega_S \ni (\omega,\varpi) \mapsto B_{S^\Psi_t(\varpi)}(\omega) \in \Rd
$$
is called subordinate Brownian motion under $\pro {S^\Psi}$, and
$$
\ex^x_{\mathbb P_S \times \mathbb P_W} [e^{iu \cdot B_{S^\Psi_t}}] = e^{t\Psi(|u|^2)}, \quad t > 0, \, u \in \Rd,
$$
holds. Except for the trivial case generated by $\Psi(u) = u$ giving Brownian motion, every subordinate Brownian motion
is a jump L\'evy process, satisfying the strong Markov property. For simplicity, we will denote a subordinate Brownian
motion by $\pro X$, its probability measure for the process starting at $x \in \Rd$ by $\mathbb P^x$, and expectation
with respect to this measure by $\ex^x$. The subordination procedure gives, in particular, the expression
\begin{equation*}
\label{subord}
\Prob (X_t \in A) = \int_0^\infty \Prob_W(B_s \in A)\Prob_S(S_t \in \D s),
\end{equation*}
for every measurable set $A$. For a detailed discussion of Bernstein functions, subordinators and subordinate Brownian
motion we refer to \cite{SSV}.

By subordination it also follows that the infinitesimal generator of subordinate Brownian motion with a subordinator 
corresponding to a Bernstein function $\Psi \in \cB_0$ is the pseudo-differential operator $-\Psi(-\Delta)$ defined by 
the Fourier multiplier
$$
\widehat{(\Psi(-\Delta) f)}(y) = \Psi(|y|^2)\widehat f(y), \quad y \in \Rd, \; f \in  \Dom(\Psi(-\Delta)),
$$
with domain $\Dom(\Psi(-\Delta))=\big\{f \in L^2(\Rd): \Psi(|\cdot|^2) \widehat f \in L^2(\Rd) \big\}$. By general
arguments it can be seen that $-\Psi(-\Delta)$ is a negative, self-adjoint operator with core $C_{\rm c}^\infty(\Rd)$.

In what follows, we will use the Hartman-Wintner condition
\begin{equation}\label{HW}
\lim_{|u|\to \infty}\frac{\Psi(\abs{u}^2)}{\log\abs{u}}=\infty.
\end{equation}
It is known that under this condition the subordinate Brownian motion $\pro X$ has a bounded continuous transition
probability density $q_t(x, y)=q_t(x-y)$, and $q_t(\cdot)$ is radially decreasing, see \cite{KS13}.

\section{\textbf{Maximum principles}}
\subsection{Weak solution and stochastic representation}
Let $\cD$ be a bounded connected domain in $\Rd$, $V\in \cC(\bar\cD)$, and consider the first exit time
\begin{equation*}
\label{exit}
\uptau_\cD = \inf\{t > 0: \, X_t \not\in \cD\}
\end{equation*}
of subordinate Brownian motion $\pro X$ from $\cD$. The killed Feynman-Kac semigroup in $\cD$ is given by
\begin{equation}
\label{killedV}
T^{\cD,V}_t f(x) = \ex^x[e^{-\int_0^t V(X_s)ds}f(X_t)\Ind_{\{\uptau_D > t\}}], \quad x \in \cD, \, t > 0, \, f\in L^2(\cD).
\end{equation}
It is shown in \cite[Lem.~3.1]{BL17a} that under condition \eqref{HW} $T_t^{\cD,V}$ is a Hilbert-Schmidt operator on $L^2(\cD)$
for every $t > 0$. Also, it is a strongly continuous semigroup in $L^2(\cD)$ with generator $-H^{\cD, V}$, where
$$
H^{\cD, V} =\Psidel + V,
$$
defined in form sense. This implies that the spectrum of $H^{\cD, V}$ is purely discrete, and there exists a countable
set $\seq \lambda$ of eigenvalues of finite multiplicity each, and corresponding square integrable eigenfunctions $\seq
\varphi$ of $H^{\cD, V}$, which form an orthonormal basis in $L^2(\cD)$. Furthermore, the eigenfunctions are bounded
continuous functions in $\cD$. Assuming that $V\geq 0$, it follows that
$$0
<\lambda_1<\lambda_2\leq \ldots
$$
Every $T^{\cD, V}_t$ is an integral operator, with integral kernel given by the eigenfunction expansion
$$
T^{\cD, V}(t, x, y)=\sum_{n=1}^\infty e^{-\lambda_n t} \varphi_n(x)\varphi_n(y), \quad t>0, \quad x, y\in\cD.
$$
These properties hold for the $V \equiv 0$ case as well, and we will denote $H^{\cD, 0} = H^\cD$. For further details
on general facts we refer to \cite[Ch.~XIX, Th.~6.2, Cor.~6.3]{GGK}.

Let $\{\lambda_{1}^0<\lambda_{2}^0\leq \ldots\}$ be the eigenvalues of the operator $H^{\cD}= \Psidel$. Since
$V\geq 0$, we can see by the min-max principle that
\begin{equation}
\label{compi}
\lambda_{n}\geq \lambda_n^0 -\norm{V}_\infty, \quad \; n\in\NN.
\end{equation}
Indeed, note that for any $f\in L^2(\cD)$
\begin{align}\label{Add1}
&\int_\cD \left(\frac{1}{t}\ex^x\left[(e^{-\int_0^t V(X_s)\, \D{s}}-1)f(X_t)\Ind_{\{\uptau_\cD>t\}}\right]-
\ex^x\left[V(X_t)f(X_t)\Ind_{\{\uptau_\cD>t\}}\right]\right)^2
\D{x}\nonumber
\\
& \quad =
\int_\cD \ex^x\left[\frac{1}{t}(e^{-\int_0^t V(X_s)\, \D{s}}-1-t V(X_t))f(X_t)\Ind_{\{\uptau_\cD>t\}}\right]^2\D{x}\nonumber
\\
&\quad\leq
\int_\cD \ex^x\left[f^2(X_t)\Ind_{\{\uptau_\cD>t\}}\right]\ex^x[\Xi^2_t]\D{x},
\end{align}
where
$$
\Xi_t=\frac{1}{t}\left(e^{-\int_0^t V(X_s)\, \D{s}}-1-t V(X_t)\right).
$$
Let $\varepsilon>0$ be arbitrary, and take a compactly supported continuous function $\psi\in L^2(\cD)$ such that
$\int_{\cD}\abs{f^2(y)-\psi^2(y)}\, \D{y}<\varepsilon$. Note that $\Xi_t$ is bounded uniformly in $t>0$ almost surely, and
$\ex^x[\Xi_t]\to 0$ as $t\to 0$ for all $x\in\cD$. Therefore, by the dominated convergence theorem we have
$$
\int_\cD \ex^x\left[\psi^2(X_t)\Ind_{\{\uptau_\cD>t\}}\right]\ex^x[\Xi^2_t]\D{x}\to 0, \quad \text{as}\quad t\to 0.
$$
On the other hand
\begin{align*}
\int_\cD \ex^x\left[\left(f^2(X_t)-\psi^2(X_t)\right)\Ind_{\{\uptau_\cD>t\}}\right]\ex^x[\Xi^2_t]\D{x}
& \leq \kappa\, \int_\cD \left[\int_{\cD}\abs{f^2(y)-\psi^2(y)} q_t(x-y)\, \D{y}\right]\D{x}
\\
&\leq \kappa \int_{\cD}\abs{f^2(y)-\psi^2(y)} \D{y}=\kappa\varepsilon,
\end{align*}
with a constant $\kappa$ dependent on $\norm{V}_\infty$, where $q_t$ is the transition density of $\pro{X}$ at time $t$.
Hence \eqref{Add1} tends to zero as $t\to 0$. Clearly, this implies
$$
\frac{1}{t}\ex^x\left[(e^{-\int_0^t V(X_s)\, \D{s}}-1)f(X_t)\Ind_{\{\uptau_\cD>t\}}\right]\to V(x) f(x)
\quad \text{as}\quad t\to 0,
$$
in $L^2(\cD)$. Hence $\Dom(H^{\cD, V})=\Dom(H^{\cD})$, and for every $f\in \Dom(H^{\cD})$ we have
$$
H^{\cD, V} f=H^{\cD} f + V f.
$$
Let now $L_n$ be any $n$-dimensional subspace of $\Dom (H^{\cD, V})=\Dom(H^\cD)$. By the min-max representation of the
eigenvalues it the follows that
\begin{align*}
\lambda_n & = \inf_{L_n}\, \sup\left\{\langle H^{\cD, V}f, f\rangle\; :\; f\in L_n \; \text{and}\; \norm{f}_{L^2(\cD)}=1\right\}
\\
&=
\inf_{L_n}\,\sup\left\{\langle H^{\cD}f, f\rangle + \langle f V, f\rangle\; :\; f\in L_n \; \text{and}\; \norm{f}_{L^2(\cD)}=1\right\}
\\
&\geq
\inf_{L_n}\, \sup\left\{\langle H^{\cD}f, f\rangle  :\; f\in L_n \; \text{and}\; \norm{f}_{L^2(\cD)}=1\right\}-\norm{V}_\infty
\\
&=
\lambda^0_n -\norm{V}_\infty,
\end{align*}
which shows \eqref{compi}.

\medskip

To define a mild or weak solution of the Cauchy problem \eqref{E1.1} we will use the operators
$$
S_t\psi = \sum_{n\geq 1} E_{\alpha, 1}(-\lambda_n t^\alpha) \langle \varphi_n, \psi\rangle \varphi_n,
$$
and
$$
K_t\psi = \sum_{n\geq 1} t^{\alpha-1} E_{\alpha, \alpha}(-\lambda_n t^\alpha) \langle \varphi_n, \psi\rangle \varphi_n,
$$
where the pointed brackets mean scalar product in $L^2(\cD)$ and
$$
E_{\alpha, \beta}(z) = \sum_{n\geq 0}\frac{z^n}{\Gamma(\alpha k+ \beta)}, \quad z\in\mathbb{C},
$$
denotes the Mittag-Leffler function (see \cite{GKMR}). Recall from \cite[p.~35]{P99} that for a suitable constant $C =
C(\alpha, \beta)$, the estimate
\begin{equation}\label{Add4}
\abs{E_{\alpha, \beta}(-s)}\leq \frac{C}{1+s}, \quad s\geq 0,
\end{equation}
holds for $\alpha\in(0, 2)$ and $\beta \in \RR$. Thus we have for every $t\geq 0$ that
$$
\left(\sup_{n\in \NN} \abs{E_{\alpha, 1}(-\lambda_n t^\alpha)}\right) \vee
\left(\sup_{n\in \NN} t^{\alpha-1} \abs{E_{\alpha, \alpha}(-\lambda_n t^\alpha)}\right)<\infty,
$$
ensuring that $S_t$ and $K_t$ are well-defined linear operators on $L^2(\cD)$ for all $t>0$.

To study the existence of a weak solution, we define powers of the operator $H^{\cD, V}$. For
$\gamma\in\RR$ let
$$
H^{\cD, V}_\gamma \psi =\sum_{n=1}^\infty \lambda_n^\gamma \langle \varphi_n, \psi\rangle \varphi_n,
\quad \psi\in \Dom(H^{\cD, V}_\gamma)
$$
with
$$
\Dom(H^{\cD, V}_\gamma)= \left\{\psi\in L^2(\cD)\; :\; \sum_{n=1}^\infty \lambda_n^{2\gamma}
\langle \varphi_n, \psi\rangle^2 < \infty \right\} := \sH_\gamma.
$$
Note that $\sH_\gamma$ is a Hilbert space with norm
$$
\norm{\psi}_{\sH_\gamma}= \left(\sum_{n=1}^\infty \lambda_n^{2\gamma}\langle \varphi_n, \psi\rangle^2\right)^{\nicefrac{1}{2}}.
$$
We also have for $\gamma>0$ that
$$
\Dom(H^{\cD, V}_\gamma)\subset L^2(\cD)\simeq (L^2(\cD))^\prime\subset \left(\Dom(H^{\cD, V}_\gamma)\right)^\prime.
$$
For notational economy we set $\left(\Dom(H^{\cD, V}_\gamma)\right)^\prime=\Dom(H^{\cD, V}_{-\gamma}) =:\sH_{-\gamma}$.
For every $f\in \sH_{-\gamma}$ we denote its action on $\psi\in\cH_{\gamma}$ by $_{-\gamma}\langle f, \psi\rangle_{\gamma}$.
The space $\sH_{-\gamma}$ is again a Hilbert space with norm
$$
\norm{f}_{\sH_{-\gamma}}=\left(\sum_{n=1}^\infty \lambda_n^{-2\gamma}\,
_{-\gamma}\langle \varphi_n, f\rangle_\gamma^2\right)^{\nicefrac{1}{2}}.
$$

Now we define weak solutions in the spirit of \cite{SY11}. Throughout this paper we use Caputo-time derivative as defined
in the Introduction.
\begin{definition}\label{D3.1}
We say that $\varphi$ is a weak solution of \eqref{E1.1} if
\begin{enumerate}
\item[(1)]
$\varphi(t, \cdot)\in \Dom(H^{\cD, V})$ for almost every $t\in (0, T)$, and $\partial^\alpha\varphi\in L^2_{\rm loc}((0, T), L^2(\cD))$,
\vspace{0.1cm}
\item[(2)]
the equality
\begin{equation}\label{ED2.1A}
\partial^\alpha_t \varphi + \Psidel\varphi + V\varphi = F(t, \cdot)
\end{equation}
holds in $L^2(\cD)$ for almost every $t\in (0, T)$,
\vspace{0.1cm}
\item[(3)]
for a $\gamma>0$ we have $\varphi\in \cC([0, T]; \sH_{-\gamma})$ with
$$
\lim_{t\to 0}\norm{\varphi(t, \cdot)-\varphi_0}_{{\sH}_{-\gamma}}=0\,.
$$
\end{enumerate}
\end{definition}

Our first goal in this section is to establish existence of a weak solution for \eqref{E1.1}. We will use the following assumption
on the Bernstein function.
\begin{assumption}\label{Ass2.1}
There exists $\beta\in (0, 1]$ such that
$$
\liminf_{\abs{u}\to\infty}\, \frac{\Psi(u)}{u^\beta}>0\,.
$$
\end{assumption}

\begin{theorem}[\textbf{Existence and uniqueness of weak solution}]
\label{exuni}
Let Assumption~\ref{Ass2.1} hold and $\cD$ be a bounded domain in $\Rd$ satisfying the exterior cone condition. If
$F\in L^\infty((0, T); L^2(\cD))$ and $\varphi_0\in L^2(\cD)$, then
\begin{equation}\label{E1.2}
\varphi(t, x) = S_t \varphi_0(x) + \int_0^t K_{t-s} F(s, x)\, \D{s}, \quad x \in \cD, \; t>0,
\end{equation}
is the unique weak solution of \eqref{E1.1} such that $\varphi\in L^2((0, T); \sH_1)$ and $\partial^\alpha_t\varphi
\in L^2((0, T)\times \cD)$. Moreover, we have
$$
\lim_{t\to 0}\norm{\varphi(t, \cdot)-\varphi_0}_{{\sH}_{-\gamma}}=0,
$$
for every $\gamma>\frac{d}{4\beta}-1$.
\end{theorem}

\begin{proof}
The arguments below are inspired by \cite{SY11} and some details are left to the reader to avoid repetitions.

\vspace{0.1cm}
\noindent
\emph{Step 1}: First consider $F=0$. In this case we have
$$
\varphi(t, x) = \sum_{n=1}^\infty \langle \varphi_n, \varphi_0\rangle E_{\alpha, 1}(-\lambda_n t^\alpha)\varphi_n(x).
$$
By similar arguments as in \cite[Th.~2.1(i)]{SY11} we obtain
$$
\lim_{t\to 0}\norm{\varphi(t, \cdot)-\varphi}_{L^2(\cD)}=0.
$$

\medskip
\noindent
\emph{Step 2}: Next consider $\varphi_0=0$.  In this case we have
$$
\varphi(t, x) = \sum_{n=1}^\infty \left[\int_0^t  (t-s)^{\alpha-1} E_{\alpha, \alpha}(-\lambda_n (t-s)^\alpha)
\langle \varphi_n, F(s, \cdot)\rangle\, \D{s}\right] \varphi_n(x).
$$
As observed in \cite{SY11}, it follows that
\begin{align}\label{ET2.1A}
&\partial^\alpha_t \left[\int_0^t  (t-s)^{\alpha-1} E_{\alpha, \alpha}(-\lambda_n (t-s)^\alpha)
\langle \varphi_n, F(s, \cdot)\rangle\, \D{s}\right]\nonumber
\\
&\quad =
-\lambda_n \left[\int_0^t  (t-s)^{\alpha-1} E_{\alpha, \alpha}(-\lambda_n (t-s)^\alpha)
\langle \varphi_n, F(s, \cdot)\rangle\, \D{s}\right] + \langle F(t, \cdot), \varphi_n\rangle.
\end{align}
The arguments in \cite[Th.~2.2]{SY11} also give that
$$
\norm{\partial^\alpha_t \varphi}^2_{L^2((0, T)\times\cD)} + \norm{H^{\cD, V}\varphi}_{L^2((0, T)\times\cD)}
\;\leq \kappa \norm{F}_{L^2((0, T)\times\cD)},
$$
with a constant $\kappa$, independently of $F$. Hence, applying \eqref{ET2.1A} we see that \eqref{ED2.1A} holds.
Next we show that
\begin{equation}\label{ET2.1B}
\lim_{t\to 0}\norm{\varphi(t, \cdot)}_{\sH_{-\gamma}}=0.
\end{equation}
By a similar type of calculation as in \cite[p.~434]{SY11} we obtain
\begin{equation}\label{ET2.1C}
\norm{\varphi(t, \cdot)}_{\sH_{-\gamma}}\leq \kappa_1 \norm{F}_{L^\infty((0, T); L^2(\cD))}
\sum_{n=1}^\infty \frac{1}{\lambda_n^{2\gamma + 2}} (1-E_{\alpha, 1}(-\lambda_n t^\alpha)),
\end{equation}
for a constant $\kappa_1$. By the assumption on $\cD$ and \cite[Th.~4.5]{Chen-Song}, there exists $\delta\in(0,1)$
satisfying
$$
\lambda_n^0 \geq \delta \Psi(\lambda_{n, \mathrm{Lap}}),
$$
where $\lambda_{n, \mathrm{Lap}}$ denotes the eigenvalues of the Laplacian with Dirichlet boundary condition in $\cD$,
arranged in increasing order and including multiplicity. On the other hand, from \cite{CH} we have $\lambda_{n,\mathrm{Lap}}
\gtrsim n^{\nicefrac{2}{d}}$. Combining this with \eqref{compi} and \eqref{ET2.1C}, and using Assumption~\ref{Ass2.1}, we
find
$$
\norm{\varphi(t, \cdot)}_{\sH_{-\gamma}}\leq \kappa_2 \norm{F}_{L^\infty((0, T); L^2(\cD))}
\sum_{n=1}^\infty \frac{1}{n^{\frac{2\beta}{d}(2\gamma + 2)}} (1-E_{\alpha, 1}(-\lambda_n t^\alpha)),
$$
with a constant $\kappa_2$. By our assumption on $\gamma$ we have $\frac{2\beta}{d}(2\gamma + 2)>1$. Since
$1-E_{\alpha, 1}(-\lambda_n t^\alpha)$ is bounded, see \cite[Lem.~3.1]{SY11}, and tends to zero as $t\to 0$, we
readily obtain \eqref{ET2.1B}.

Finally, to show uniqueness a reasoning similar to \cite[pp.432-433]{SY11} can be used.
\end{proof}

\begin{remark}
It is easily seen from \eqref{E1.2} and \eqref{Add4} that for $F=0$ and $\varphi\in L^2(\cD)$ we obtain the decay
rate
$$
\norm{\varphi(t, \cdot)}_{L^2(\cD)}\leq \frac{C}{1+ \lambda_1\, t^\alpha}, \quad t>0,
$$
with a constant $C > 0$, which can be compared with \cite[Cor.~2.6]{SY11}.
\end{remark}

In the remaining part of this paper our main tool will be formula \eqref{E1.2} and its probabilistic representation
which we derive next. Let $\pro \xi$ be a stable subordinator with exponent $\alpha \in (0,1)$, and denote by $g_t$
its smooth transition probability density at time $t$. The function $g_t$ is bounded for every $t>0$, see \cite{K81}.
Moreover, by \cite[Th.~1.1]{W07} there exists a constant $c_1 > 0$ such that
\begin{equation}\label{gbound}
g_1(x)\leq  \frac{c_1}{(1+ x)^{1+\alpha}},\quad x > 0.
\end{equation}
By usual scaling it is also known that
$$
g_t(u) = t^{-\nicefrac{1}{\alpha}} g_1(t^{-\nicefrac{1}{\alpha}} u), \quad t>0\,.
$$
Let $\pro\eta$ be the inverse of $\pro\xi$, i.e., the process
$$
\eta_t =\inf \{ s > 0\; :\; \xi_s> t\}, \quad t > 0.
$$
It is known  \cite{MNV} that $\pro\eta$ has density
\begin{equation}\label{E2.8}
\eta_t(\D{u})= t \alpha^{-1} u^{-1-\nicefrac{1}{\alpha}} g_1(u^{-\nicefrac{1}{\alpha}}t)\, \D{u}, \quad t>0.
\end{equation}
We choose the process $\pro\eta$ to be independent of the subordinate Brownian motion $\pro X$, and continue to make
the assumption $V\geq 0$ without specifying it explicitly.

\begin{theorem}[\textbf{Probabilistic representation}]
\label{T2.2}
Let $\varphi$ satisfy \eqref{E1.2} with given $\varphi_0\in L^2(\cD)$ and $F\in L^2((0, T)\times \cD)$. Then we have
\begin{equation}\label{ET2.2A}
\varphi(t,x)= \int_0^\infty T^{\cD, V}_u\varphi_0(x)\, \eta_t(\D{u}) +
\int_0^t \left[\int_0^\infty  l^{-\nicefrac{1}{\alpha}} g_1(l^{-\nicefrac{1}{\alpha}}(t-s)) T_l^{\cD, V}F(s, x)\;
\D{l}\right]\D{s}\,,
\end{equation}
where $\semi {T^{\cD, V}}$ is the semigroup defined by \eqref{killedV}.
\end{theorem}
\begin{proof}
First we show that
\begin{equation}\label{ET2.2B}
S_t\varphi_0(x)=\ex^x\left[ e^{-\int_0^{\eta_t} V(X_s)\, \D{s}} \varphi_0(X_{\eta_t}) \Ind_{\{\uptau_\cD>\eta_t\}}\right],
\end{equation}
which gives the first term at the right hand side of \eqref{ET2.2A}. Denote Laplace transform by $\cL$, and the
transition probability density of $\pro X$ at time $t>0$ by $q_t$. We use the fact that the Laplace transform of $E_{\alpha, 1}
(-\lambda t^\alpha)$ is $\frac{s^{\alpha-1}}{s^\alpha+\lambda}$. Also, for every $\psi \in L^2(\cD)$, extending it to $\Rd$
by zero, we have
\begin{align}\label{ET2.2C}
\norm{T^{\cD, V}_l\psi}^2_{L^2(\cD)} &\leq \int_{\cD} \left[\ex^x[\psi(X_l)\Ind_{\{\uptau_\cD> l\}}]\right]^2\,
\D{x}\nonumber
\\
&\leq  \int_{\cD} \ex^x[\psi^2 (X_l)]\Prob^x(\uptau_\cD>l) \D{x}\nonumber
\\
&\leq \left[\sup_{\cD} \Prob^x(\uptau_\cD>l)\right] \int_{\Rd} \int_{\Rd} \psi^2 (y)q_l(x, y)\D{y}\, \D{x}
\leq\norm{\psi}^2_{L^2(\cD)}.
\end{align}
Thus we have
\begin{align*}
\cL(S_\cdot \varphi_0)(s)
&=
\sum_{n=1}^\infty  \frac{s^{\alpha-1}}{s^\alpha+\lambda_n} \langle \varphi_n, \varphi_0\rangle \varphi_n,
\\
&=
\sum_{n=1}^\infty s^{\alpha-1}\left[\int_0^\infty e^{-(s^\alpha+\lambda_n)l}\, \D{l}\right]
\langle\varphi_n,\varphi_0\rangle \varphi_n,
\\
&=
\int_0^\infty s^{\alpha-1} e^{-s^\alpha\, l}
\left[\sum_{n=1}^\infty e^{-\lambda_n l}\langle \varphi_n, \varphi_0\rangle \varphi_n\right]\, \D{l}
\\
&=
\int_0^\infty  s^{\alpha-1} e^{-s^\alpha\, l} T^{\cD, V}_l\varphi_0(x)\, \D{l}.
\end{align*}
The above equality should be understood in $L^2(\cD)$, which is justified by \eqref{ET2.2C}.

Recall the equality
$$
s^{\alpha-1} e^{-s^\alpha\, l}=\frac{1}{\alpha}\int_0^\infty u e^{-s u} g_1(u l^{-\nicefrac{1}{\alpha}})
l^{-1-\nicefrac{1}{\alpha}}\, \D{u},
$$
see \cite{MBSB}. Combining it with the above gives
$$
\cL(S_\cdot \varphi_0)(s)=
\int_0^\infty e^{-s u}\left[\frac{1}{\alpha}\int_0^\infty u  g_1(u l^{-\nicefrac{1}{\alpha}})
l^{-1-\nicefrac{1}{\alpha}} T^{\cD, V}_l\psi_0(x)\, \ \D{l}\right]\, \D{u}
$$
Comparing the Laplace transforms on both sides and using \eqref{E2.8}, we obtain \eqref{ET2.2B}. We note
that a similar argument appeared in \cite{MNV} in another context.

Next we calculate the second term at the right hand side of \eqref{ET2.2A}. We will use that for $f_\lambda(t)=
t^{\alpha-1} E_{\alpha, \alpha}(-\lambda t^\alpha)$, $\lambda>0$, the expression
$$
\cL(f_\lambda)(s)= \frac{1}{s^\alpha+\lambda},\quad s>\lambda^{\nicefrac{1}{\alpha}},
$$
holds, see \cite[p.~312]{GKMR}. Let $\psi\in \cC_c(\cD)$, i.e., a continuous function with compact support. Define
$$
K^m_t\psi = \sum_{n= 1}^m t^{\alpha-1} E_{\alpha, \alpha}(-\lambda_n t^\alpha) \langle \varphi_n, \psi\rangle \varphi_n,
$$
and
$$
T_l^{\cD, V, m}\psi = \sum_{n= 1}^m e^{-\lambda_n l}\langle \varphi_n, \psi\rangle \varphi_n.
$$
Note that $K^m\psi$ and $T_l^{\cD, V, m}\psi$ are continuous functions in $(t, x)\in(0, \infty)\times\cD$. Thus for
every $s^\alpha> \max\{\lambda_1,\ldots, \lambda_m\}$ we have
\begin{align*}
\cL(K^m_\cdot\psi)(s) &=\sum_{n=1}^m \frac{1}{s^\alpha+\lambda_n} \langle \varphi_n, \psi\rangle \varphi_n
= \sum_{n= 0}^m \int_0^\infty e^{-(s^\alpha+\lambda_n)l}\D{l} \; \langle \varphi_n, \psi\rangle \varphi_n
\\
&=
\int_0^\infty e^{-s^\alpha\, l} \left[ \sum_{n=1}^m e^{-\lambda_n l}\langle \varphi_n, \psi\rangle \varphi_n\right] \D{l}
=
\int_0^\infty e^{-s^\alpha\, l} T_l^{\cD, V, m}\psi\,  \D{l}
\\
&=
\int_0^\infty \left[\int_0^\infty e^{-su} g_l(u)\,\D{u}\right] T_l^{\cD, V, m}\psi\; \D{l}
\\
&=
\int_0^\infty e^{-su} \left[\int_0^\infty  l^{-\nicefrac{1}{\alpha}} g_1(l^{-\nicefrac{1}{\alpha}}u)
T_l^{\cD, V, m}\psi\; \D{l}\right] \D{u}.
\end{align*}
Again comparing the two Laplace transforms, we get
$$
K^m_t\psi= \int_0^\infty  l^{-\nicefrac{1}{\alpha}} g_1(l^{-\nicefrac{1}{\alpha}}t) T_l^{\cD, V,m}\psi\; \D{l}.
$$
Letting $m\to\infty$ and using a denseness argument, we obtain
$$
K_t\psi= \int_0^\infty  l^{-\nicefrac{1}{\alpha}} g_1(l^{-\nicefrac{1}{\alpha}}t) T_l^{\cD, V}\psi\; \D{l},
\quad \psi\in L^2(\cD).
$$
Applying this equality to $F$ in \eqref{E1.2}, the result follows.
\end{proof}

\begin{remark}\label{R3.1}
Note that for every $\kappa>0$ and $s>0$ we have by \eqref{gbound}
\begin{align*}
\int_0^\kappa l^{-\nicefrac{1}{\alpha}} g_1(s l^{-\nicefrac{1}{\alpha}})\, \D{l}
&\leq
c_1 \int_0^\kappa l^{-\nicefrac{1}{\alpha}} \frac{1}{(1+ s l^{-\nicefrac{1}{\alpha}})^{1+\alpha}}\, \D{l}
\\
&= c_1 s^{1-\alpha} \int_0^{\kappa s^{-\alpha}} \frac{u}{(1+u^{\nicefrac{1}{\alpha}})^{1+\alpha}} \, \D{u},
\\
&\leq c_1 s^{1-\alpha} \int_0^{\infty} \frac{u}{(1+u^{\nicefrac{1}{\alpha}})^{1+\alpha}} \, \D{u} = c_2 s^{1-\alpha},
\end{align*}
with a constant $c_2$, and in the second line we used the substitution $l=s^\alpha u$. The above integral is finite since
$\alpha\in (0, 1)$. The estimate also implies that the rightmost term in \eqref{ET2.2A} is finite.
\end{remark}

As a consequence of Thoerem~\ref{T2.2} we have the following results which generalize Theorems ~1 and ~3 in \cite{LY}
substantially.
\begin{corollary}\label{C2.1}
Let $\varphi^i_0\in L^2(\cD)$ and $F_i\in L^2((0, T)\times\cD)$ for $i=1, 2$. Furthermore, assume that $\varphi^1_0
\geq \varphi^2_0$ and $F_1\geq F_2$. Then we have $\varphi^1(t, x)\geq \varphi^2(t, x)$ almost surely in $(0, T)\times\cD$.
\end{corollary}

\begin{corollary}\label{C2.2}
Let $V_i\in \cC(\bar\cD)$ and satisfy $V_1\geq V_2$. Let $\varphi_i$ denote the solution of \eqref{E1.2} with non-negative data
$(\varphi_0, F)$ and $V_i$, $i=1,2$. Then we have $\varphi_1\leq \varphi_2$, almost surely in $(0, T)\times\cD$.
\end{corollary}

\begin{proof}
This follows from the fact that $T^{\cD, V_1}_t \psi \leq T^{\cD, V_2}_t \psi$ for all $t$ whenever $\psi\geq 0$.
\end{proof}

\bigskip
\subsection{Maximum principles and an Aleksandrov-Bakelman-Pucci estimate}

Next we turn to proving maximum principles. For the remainder of the paper we assume that the domain $\cD$ is such that all its
boundary points are regular with respect to $\pro X$, i.e.,
$$
\Prob^z(\uptau_\cD=0)=1, \quad z \in \partial\cD.
$$
In \cite{BL17a} we have shown that every bounded convex set is regular with respect to subordinate Brownian motion. When $\pro X$
is an isotropic $\alpha$-stable process, any domain $\cD$ with the exterior cone condition is regular with respect to this process.

Recall the notation
$$
h \gneq 0 \quad \mbox{meaning} \quad h(x) \geq 0 \; \mbox{for all $x \in \cD$ and $h \not\equiv 0$.}
$$
\begin{theorem}[\textbf{Strong maximum principle}]\label{T2.3}
Let $\varphi_0\gneq 0$ and $F=0$, and suppose that $\cD$ has a regular boundary with respect to $\pro X$. Furthermore, assume that
$\varphi_0 \in L^\infty(\cD)$. If $\varphi$ is a weak solution satisfying \eqref{E1.2}, then $\varphi(t, \cdot)\in \cC(\bar\cD)$ and
$\varphi(t, x)>0$  for every $t>0$ and $x\in\cD$.
\end{theorem}

\begin{proof}
From \cite[Lem.~3.1]{BL17a} we know that
\begin{equation}\label{ET2.3A}
T^{\cD, V}_t \varphi_0(x) =\int_{\cD} T^{\cD, V}(t, x, y) \varphi_0(y)\, \D{y},\quad t>0,
\end{equation}
with kernel
$$
T^{\cD, V}(t, x, y)= \ex_{\Prob_S}\left[p_{S^\Psi_t} (x-y)
\ex^{x, y}_{0, S^\Psi_t}\left[e^{-\int_0^t V(B_{S^\Psi_s}) \, \D{s}} \Ind_{\{\uptau_\cD>t\}} \right]\right],
$$
where $p_t(x-y)=(4\pi t)^{-\nicefrac{d}{2}} \exp(-\frac{\abs{x-y}^2}{4t})$, and $\ex^{x, y}_{0, S^\Psi_t}[\cdot]$ denotes
expectation with respect to the Brownian bridge measure from $x$ at time $0$ to $y$ at time $s$, evaluated at random time
$s=S^\Psi_t$. It is also shown in \cite[Lem.~3.1]{BL17a} that $T^{\cD, V}(t, x, y)=T^{\cD, V}(t,  y, x)$ for $t>0$ and
$T^{\cD, V}(t, x, y)$ is continuous in $(0,\infty) \times\cD\times\cD$. This implies that $x\mapsto T^{\cD, V}_t \varphi_0(x)$ is
continuous in $\cD$ for $t>0$.
Furthermore, by \cite[Lem.~3.1(v)]{BL17a}, we see
that $x\mapsto T^{\cD, V}_t 1(x)$ is continuous in $\bar\cD$ and vanishes on the boundary. Since
$$
\abs{T^{\cD, V}_t \varphi_0(x)} \leq
\norm{\varphi_0}_{L^2(\cD)} \norm{T^{\cD, V}(t, x, \cdot)}^{\nicefrac{1}{2}}_{L^\infty(\cD)}
\abs{T^{\cD, V}_t 1(x)}^{\nicefrac{1}{2}},
$$
it follows that $x\mapsto T^{\cD, V}_t \varphi_0(x)$ is continuous in $\bar\cD$ for every $t>0$ and the function vanishes on $\partial\cD$. Again, for every $t>0$,
$$
\abs{T^{\cD, V}_t \varphi_0(x)}\leq \norm{\varphi_0}_{L^\infty(\cD)}T^{\cD, V}_t 1(x)
\leq \norm{\varphi_0}_{L^\infty(\cD)} \Prob^x(\uptau_\cD>t).
$$
Since
$$
\varphi(t, x) = \int_0^\infty T^{\cD, V}_u \varphi_0(x)\, \eta_t(\D{u}), \quad t>0\,
$$
by \eqref{ET2.2A}, it follows from the dominated convergence theorem that $\varphi(t, \cdot)\in\cC(\bar\cD)$. The second
claim follows by \eqref{ET2.3A} and the observation that $T^{\cD, V}(t, x, y)>0$ for $t>0$ and $x, y\in\cD$. Indeed, note
that for every $\kappa>0$, by the support theorem of Brownian paths there are paths inside $\cD$ starting from $x$ at time
$0$ and ending at $y$ at time $\kappa$ with positive probability. This implies that for a given $S^\Psi_t$,
$$
\ex^{x, y}_{0, S^\Psi_t}\left[e^{-\int_0^t V(B_{S^\Psi_s}) \, \D{s}} \Ind_{\{\uptau_\cD>t\}} \right]>0
$$
holds which, in turn, implies $T^{\cD, V}(t, x, y)>0$.
\end{proof}

\begin{remark}
A much weaker version of Theorem~\ref{T2.2} has been obtained in \cite[Th.~1.1]{LRY} and \cite[Cor.~4]{LY} where $H^{\cD, V}$
was given by a divergence form elliptic operator. A strong maximum principle was conjectured in \cite{LRY}. Our result applies
to a much larger class of non-local operators and establishes the full strong maximum principle.
\end{remark}

We also have a stability result.
\begin{theorem}[\textbf{Stability of solutions}]
\label{stab}
Suppose that $V_i\in\cC(\bar\cD)$, $V_i\geq 0$, $\varphi_0^i\in L^2(\cD)$ and $F_i\in L^\infty((0, T); L^2(\cD))$ for $i=1,2$. Let
$\varphi_i$ be the corresponding weak solution given by \eqref{ET2.2A}. Then for a constant $C = C(T, d, \Psi, \cD, \alpha)$ we have
\begin{align}\label{ET3.4A}
\esssup_{(0, T)}\; \norm{\varphi_1(t, \cdot)-\varphi_2(t, \cdot)}_{L^2(\cD)}
&\leq
C\Bigl[ \mathfrak{A}\,\norm{V_1-V_2}_{L^\infty(\cD)}  + \norm{\varphi^1_0-\varphi^2_0}_{L^2(\cD)} \nonumber
\\
&\, \qquad + \esssup_{(0, T)} \norm{F_1(s, \cdot)-F_2(s, \cdot)}_{L^2(\cD)}\, \D{s}\Bigr]\,,
\end{align}
where
$$
\mathfrak{A}=
\max\left\{\norm{\varphi^1_0}_{L^2(\cD)}, \, \norm{\varphi^2_0}_{L^2(\cD)}, \, \esssup_{(0, T)} \norm{F_1(s, \cdot)}_{L^2(\cD)},
\, \esssup_{(0, T)} \norm{F_2(s, \cdot)}_{L^2(\cD)} \right\}.
$$
\end{theorem}

\begin{proof}
We begin with the following estimate. Consider $\psi_i\in L^2(\cD)$, $i=1,2,$ and extend them outside of $\cD$ by zero. Then
for every $x\in\cD, t>0,$ we note that
\begin{align*}
\abs{T^{\cD, V_1}_t\psi_1(x)-T^{\cD, V_2}_t\psi_2(x)} &\leq \ex^x\left[\abs{e^{-\int_0^t V_1(X_s)\, \D{s}} -
e^{-\int_0^t V_2(X_s)\, \D{s}}}\abs{\psi_1(X_t)} \Ind_{\{\uptau_\cD> t\}}\right]
\\
&\, \qquad + \ex^x\left[e^{-\int_0^t V_2(X_s)\, \D{s}}\abs{\psi_1(X_t)-\psi_2(X_t)} \Ind_{\{\uptau_\cD> t\}}\right]
\\
&\leq  t\norm{V_1-V_2}_{L^\infty(\cD)} \ex^x[\abs{\psi_1(X_t)} \Ind_{\{\uptau_\cD> t\}}] +
\ex^x\left[\abs{\psi_1(X_t)-\psi_2(X_t)} \Ind_{\{\uptau_\cD> t\}}\right]
\\[2mm]
&\leq t \norm{V_1-V_2}_{L^\infty(\cD)}\, \big(\sup_{x\in\cD}\Prob^x(\uptau_\cD>t)\big)^{\nicefrac{1}{2}}\,
\ex^x[\abs{\psi_1(X_t)}^2]^{\nicefrac{1}{2}}
\\
&\, \qquad+ \big(\sup_{x\in\cD}\Prob^x(\uptau_\cD>t)\big)^{\nicefrac{1}{2}}\, \ex[\abs{\psi_1(X_t)-\psi_2(X_t)}^2]^{\nicefrac{1}{2}},
\end{align*}
where in the second inequality we used the fact that $x\mapsto e^{-x}$ is Lipschitz continuous in $[0, \infty)$ and $V_2\geq 0$.
Thus, following a calculation similar to \eqref{ET2.2C}, we obtain
\begin{align}\label{ET3.4B}
\int_{\cD} \abs{T^{\cD, V_1}_t\psi_1(x)-T^{\cD, V_2}_t\psi_2(x)}^2\, \D{x}
& \leq
\kappa_1 t^2 \big(\sup_{x\in\cD}\Prob^x(\uptau_\cD>t)\big)\norm{V_1-V_2}^2_{L^\infty(\cD)}\, \norm{\psi_1}^2_{L^2(\cD)}\nonumber
\\
&\, \qquad + \big(\sup_{x\in\cD}\Prob^x(\uptau_\cD>t)\big)\, \norm{\psi_1-\psi_2}^2_{L^2(\cD)} \,,
\end{align}
with a constant $\kappa_1$. By \cite[Lem.~3.1]{BL17b}, for every $k\in\NN$ there exists a constant $c_k$, dependent on $d, k$,
satisfying
\begin{equation}\label{ET3.4C}
\sup_{\cD}\ex^x[\uptau^k_\cD]\leq \frac{c_k}{(\Psi([\diam \cD]^{-2}))^k}\,,
\end{equation}
where $\diam \cD$ denotes the diameter of $\cD$. Therefore choosing $k=5$ in \eqref{ET3.4C} and putting it in \eqref{ET3.4B}, with
a constant $\kappa_2$ we have
\begin{align}\label{ET3.4D}
& \int_{\cD} \abs{T^{\cD, V_1}_t\psi_1(x)-T^{\cD, V_2}_t\psi_2(x)}^2\, \D{x} \nonumber \\
& \qquad \leq \kappa_2 t^2 (1\wedge t^{-5})
\norm{V_1-V_2}^2_{L^\infty(\cD)}\, \norm{\psi_1}^2_{L^2(\cD)} + \kappa_2 (1\wedge t^{-5})\, \norm{\psi_1-\psi_2}^2_{L^2(\cD)} \,.
\end{align}
We use \eqref{ET3.4D} to estimate the left hand side of \eqref{ET3.4A}. Our main ingredient is formula \eqref{ET2.2A}. By
Minkowski's integral inequality we note that
\begin{align}\label{ET3.4E}
&\left[\int_{\cD}\left|\int_0^\infty T^{\cD, V_1}_u\varphi^1_0(x)- T^{\cD, V_1}_u\varphi^2_0(x) \,
\eta_t(\D{u})\right|^2 \, \D{x}\right]^{\nicefrac{1}{2}}
\nonumber
\\
&\quad \leq \int_0^\infty \left[\int_{\cD}\abs{T^{\cD, V_1}_u\varphi^1_0(x)- T^{\cD, V_1}_u\varphi^2_0(x)}^2 \, \D{x}\right]^{\nicefrac{1}{2}}
\eta_t(\D{u})
\nonumber
\\
&\quad \leq \sqrt{\kappa_2} \int_0^\infty \left[u(1\wedge u^{-\nicefrac{5}{2}}) \mathfrak{A}\norm{V_1-V_2}_{L^\infty(\cD)} +
(1\wedge u^{-\nicefrac{5}{2}})\norm{\varphi^1_0-\varphi^2_0}_{L^2(\cD)}
\right] \eta_t(\D{u})\nonumber
\\
&\quad \leq \kappa_3\, \left(\mathfrak{A}\norm{V_1-V_2}_{L^\infty(\cD)} + \norm{\varphi^1_0-\varphi^2_0}_{L^2(\cD)}\right),
\end{align}
with a constant $\kappa_3$, where in the third line we used \eqref{ET3.4D}. Now we compute the difference for the rightmost
term in \eqref{ET2.2A}. We again apply Minkowski's integral inequality to get
\begin{align}\label{ET3.4F}
&\left[\int_{\cD} \left[ \int_0^t \int_0^\infty  l^{-\nicefrac{1}{\alpha}} g_1(l^{-\nicefrac{1}{\alpha}}(t-s))
\abs{T_l^{\cD, V_1}F_1(s, x)-T_l^{\cD, V_2} F_2(s,x)}\; \D{l}\,\D{s} \right]^2 \, \D{x}\right]^{\nicefrac{1}{2}}\nonumber
\\
&\quad \leq
\int_0^t \int_0^\infty  l^{-\nicefrac{1}{\alpha}} g_1(l^{-\nicefrac{1}{\alpha}}(t-s))
\left[\int_{\cD}\abs{T_l^{\cD, V_1}F_1(s, x)-T_l^{\cD, V_2} F_2(s,x)}^2\, \D{x}\right]^{\nicefrac{1}{2}}\, \D{l}\,\D{s}\nonumber
\\
&\quad \leq
\int_0^t \int_0^\infty  l^{-\nicefrac{1}{\alpha}} g_1(l^{-\nicefrac{1}{\alpha}}(t-s))
\left[\int_{\cD}\abs{T_l^{\cD, V_1}F_1(s, x)-T_l^{\cD, V_2} F_2(s,x)}^2\, \D{x}\right]^{\nicefrac{1}{2}}\, \D{l}\,\D{s}\nonumber
\\
&\quad \leq
\kappa_4\, \int_0^t \int_0^\infty  l^{-\nicefrac{1}{\alpha}}g_1(l^{-\nicefrac{1}{\alpha}}(t-s))
\left[l(1\wedge l^{-\nicefrac{5}{2}}) \norm{F_1(s,\cdot)}_{L^2(\cD)}\,\norm{V_1-V_2}_{L^\infty(\cD)} + \right.\nonumber \\
& \hspace{6cm} \qquad\left. (1\wedge l^{-\nicefrac{5}{2}})\norm{F_1(s, \cdot)-F_2(s, \cdot)}_{L^2(\cD)}\right]\, \D{l}\,\D{s}
\nonumber
\\
&\quad \leq
\kappa_5\,  \left(\mathfrak{A} \norm{V_1-V_2}_{L^\infty(\cD)} + \esssup_{(0, T)} \norm{F_1(s, \cdot)-F_2(s, \cdot)}_{L^2(\cD)}\, \D{s}\right),
\end{align}
with constants $\kappa_4, \kappa_5$, where in the fourth line we used again \eqref{ET3.4D}. In \eqref{ET3.4F} we made use of
Remark~\ref{R3.1} showing that the integral in $l$ converges. Then \eqref{ET3.4A} follows by combining \eqref{ET2.2A}, \eqref{ET3.4E}
and \eqref{ET3.4F}.
\end{proof}

The next result gives an Aleksandrov-Bakelman-Pucci (ABP) estimate for the time-fractional Cauchy problem \eqref{E1.1}.
We will use a class of Bernstein functions with the following property, for an introduction see \cite{BGR14b}.
\begin{assumption}
\label{WLSC}
The function $\Psi$ is said to satisfy a weak lower scaling (WLSC) property with parameters ${\underline\mu} > 0$,
$\underline{c} \in(0, 1]$ and $\underline{\theta}\geq 0$, if
$$
\Psi(\gamma u) \;\geq\; \underline{c}\, \gamma^{\underline\mu} \Psi(u), \quad u>\underline{\theta}, \; \gamma\geq 1.
$$
\end{assumption}

\vspace{0.1cm}
\noindent
We note that the WLSC property implies that the Hartman-Wintner condition \eqref{HW} holds.

\begin{example}\label{Eg2.2}
The specific cases of $\Psi$ in Example \ref{Eg2.1} satisfy Assumption \ref{WLSC} with $\underline{\theta}=0$
and the following values:
\begin{itemize}
\item[(i)]
${\underline\mu} = \frac{\nu}{2}$
\vspace{0.1cm}
\item[(ii)]
${\underline\mu} = \frac{\nu}{2}$
\vspace{0.1cm}
\item[(iii)]
${\underline\mu} = \frac{\nu}{2} \wedge \frac{\tilde\nu}{2}$
\vspace{0.1cm}
\item[(iv)]
${\underline\mu}=\frac{\nu-\tilde\nu}{2}$
\vspace{0.1cm}
\item[(v)]
${\underline\mu}=\frac{\nu}{2}$,
\vspace{0.1cm}
\end{itemize}
where the order of cases listed above corresponds to the enumeration in Example \ref{Eg2.1}.
\end{example}

We recall the standing assumption $V\geq 0$.
\begin{theorem}[\textbf{ABP estimate}]\label{T2.4}
Let $\Psi$ satisfy Assumption \ref{WLSC}, and $\varphi$ be a weak solution of \eqref{E1.1} given by representation
\eqref{E1.2}. Furthermore, assume $\varphi^+_0\in L^\infty(\cD)$ and $F\in L^p((0, T)\times \cD)$ with $p >
\frac{d}{2\underline{\mu}} + \frac{1}{\alpha}$. Then for almost every $(t, x)\in(0, T) \times \cD$ we have
\begin{equation}\label{ET2.4A}
\varphi^+(t, x)\leq \norm{\varphi^+_0}_{L^\infty(\cD)} + C \norm{F^+}_{L^p((0, T)\times \cD)}\,,
\end{equation}
with a constant $C$, dependent on $\cD, p, d, \Psi$. (Here the plus superscript means positive part.)
\end{theorem}

\begin{proof}
As before, we denote by $q_t$ the transition density of $\pro X$. %Such density exists due to \eqref{HW}.
Extending $\varphi_0$ by $0$ outside of $\cD$, we note that for $t>0$
$$
T^{\cD, V}_t \varphi(x)\leq \ex^x\left[\varphi(X_t)\right]\leq \int_{\Rd} \varphi^+_0(y) q_t(x-y)\, \D{y}
\leq \norm{\varphi^+_0}_{L^\infty(\cD)},
$$
where the last estimate follows by the H\"{o}lder inequality. Thus also
\begin{equation}\label{ET2.4B}
\int_0^\infty T^{\cD, V}_u\varphi_0(x)\, \eta_t(\D{u})\leq \norm{\varphi^+_0}_{L^\infty(\cD)}.
\end{equation}
Taking into account \eqref{ET2.2A} and \eqref{ET2.4B}, we only need to estimate the rightmost term in \eqref{ET2.2A}.
Due to Assumption~\ref{WLSC} and \cite[Lem.~2.1]{BL17b} there exist positive constants $\kappa_1, \kappa_2$ satisfying
\begin{equation}\label{ET2.4C}
q_t(x)\leq \kappa_1 t^{-\frac{d}{2\underline{\mu}}}, \quad x\in\Rd, \; t\in (0, \kappa_2]\,.
\end{equation}
Consider $\psi\in L^p(\cD)$ with $p>\frac{d}{2\underline{\mu}} +
\frac{1}{\alpha}$, and extend it outside $\cD$ by $0$. Then for $t\in(0, \kappa_2]$ we obtain
$$
T^{\cD, V}_t \psi(x) \leq \int_{\Rd} \psi^+(y) q_t (x-y)\, \D{y}
\leq \left[\int_{\Rd} \psi^+(y)^p q_t (x-y)\, \D{y}\right]^{\nicefrac{1}{p}}
\leq \kappa_1^{\nicefrac{1}{p}} t^{-\frac{d}{2p\underline{\mu}}}\norm{\psi^+}_{L^p(\cD)},
$$
where in the last line we used \eqref{ET2.4C}. Therefore,
\begin{align}\label{ET2.4E}
\int_0^{\kappa_2} l^{-\nicefrac{1}{\alpha}}g_1(l^{-\nicefrac{1}{\alpha}}(t-s)) T^{\cD, V}_l \psi(x)\, \D{l}
& \leq \kappa_1^{\nicefrac{1}{p}}  \norm{\psi}_{L^p(\cD)}
\int_0^{\kappa_2} l^{-\nicefrac{1}{\alpha}} l^{-\frac{d}{2p\underline{\mu}}}g_1(l^{-\nicefrac{1}{\alpha}}(t-s))\, \D{l}
\nonumber
\\
&\leq \kappa_3 (t-s)^{\alpha-1-\frac{d\alpha}{2p\underline{\mu}}}\norm{\psi^+}_{L^p(\cD)},
\end{align}
with a constant $\kappa_3$. To obtain the last line, note that by \eqref{gbound} (see also Remark~\ref{R3.1}) with $\gamma
=\frac{1}{\alpha}+\frac{d}{2p\underline{\mu}}$ and $s>0$ we have
\begin{align}\label{Add3}
\int_0^{\kappa_2} l^{-\gamma} g_1(l^{-\nicefrac{1}{\alpha}}s)\, \D{l}
& \leq
c_1 \int_0^{\kappa_2} l^{-\gamma}\frac{1}{(1 + l^{-\nicefrac{1}{\alpha}}s)^{1+\alpha}} \, \D{l}\nonumber
\\
&=
s^{-\alpha\gamma + \alpha} \int_0^{\kappa_2 s^{-\alpha}}
\frac{u^{1-\frac{d}{2p\underline{\mu}}}}{(1+u^{\nicefrac{1}{\alpha}})^{1+\alpha}}\, \D{u} \nonumber \\
& \leq  s^{-\alpha\gamma + \alpha} \int_0^{\infty} \frac{u^{1-\frac{d}{2p\underline{\mu}}}}
{(1+u^{\nicefrac{1}{\alpha}})^{1+\alpha}}\, \D{u},
\end{align}
where the integral above is finite and $-\alpha\gamma + \alpha=\alpha-1-\frac{d\alpha}{2p\underline{\mu}}$.  Observe that
the assumption $p>\frac{d}{2\underline{\mu}} +\frac{1}{\alpha}$ gives
%\Leftrightarrow (\alpha-1) - \frac{d\alpha}{2p\underline{\mu}} +  1-\frac{1}{p}>0\Leftrightarrow
$(\alpha-1-\frac{d\alpha}{2p\underline{\mu}})\frac{p}{p-1}+1>0$. Thus by using Young's inequality and \eqref{ET2.4E} we obtain
that
\begin{align}\label{Add2}
\int_0^t \int_0^{\kappa_2} l^{-\nicefrac{1}{\alpha}}g_1(l^{-\nicefrac{1}{\alpha}}(t-s)) T^{\cD, V}_l F^+(s, x)\, \D{l}\, \D{s}
& \leq \kappa_3 \int_0^t (t-s)^{\alpha-1-\frac{d\alpha}{2p\underline{\mu}}} \norm{F^+(s, \cdot)}_{L^p(\cD)}\, \D{s}\nonumber
\\
&\leq \kappa_3^\prime \norm{F^+}_{L^p((0, T)\times \cD)},
\end{align}
with a constant $\kappa_3^\prime$.

Next consider $t\in(\kappa_2, \infty)$. Since for $t\geq \kappa_2$ we have
$$
\sup_{x \in \Rd} q_t(x) =\frac{1}{(2\pi)^d}\int_{\Rd} e^{-i x\cdot y} e^{-t\Psi(\abs{y}^2)} \D y
\leq
\frac{1}{(2\pi)^d}\int_{\Rd}  e^{-t\Psi(\abs{y}^2)} \D y
\leq
\frac{1}{(2\pi)^d}\int_{\Rd}  e^{-\kappa_2\Psi(\abs{y}^2)} \D y,
$$
we obtain
\begin{equation}\label{AB100}
\sup_{t\geq \kappa_2}\sup_{x, y \in \Rd}q_t(x, y)\leq q_{\kappa_2}(0).
\end{equation}
Then for $t\geq \kappa_2$ we get
\begin{align*}
T^{\cD, V}_t \psi(x)\leq \ex^x\left[\psi^+(X_t) \Ind_{\{\uptau_\cD> t\}}\right]
&\leq \ex^x\left[\left(\psi^+(X_t)\right)^p \right]^{\frac{1}{p}} \Prob^x(\uptau_\cD> t)^{\frac{p-1}{p}}
\\
&\leq q_{\kappa_2}(0)^{\frac{1}{p}} \norm{\psi^+}_{L^p(\cD)},
\end{align*}
using \eqref{AB100} in the last line. Since $\frac{1}{\alpha}>1$, we have then
\begin{equation}\label{ET2.4F}
\int_{\kappa_2}^\infty  l^{-\nicefrac{1}{\alpha}} g_1(l^{-\nicefrac{1}{\alpha}}(t-s)) T_l^{\cD, V}\psi(x)\; \D{l}
\leq \kappa_4 \norm{\psi^+}_{L^p(\cD)},
\end{equation}
with constant $\kappa_4$. From \eqref{ET2.4F} we get
$$
\int_0^{t} \int_{\kappa_2}^\infty  l^{-\nicefrac{1}{\alpha}} g_1(l^{-\nicefrac{1}{\alpha}}(t-s)) T_l^{\cD, V}F(s,x)\; \D{l}
\leq \kappa_5 \norm{F^+}_{L^p((0, T)\times \cD)},
$$
with a constant $\kappa_5$. Thus \eqref{ET2.4A} follows by using \eqref{ET2.2A}, \eqref{ET2.4B} and \eqref{Add2}.
\end{proof}

\begin{remark}
It is interesting to point out the similarity of the bound in Theorem~\ref{T2.4} and the result obtained in
\cite[Th.~3.2]{BL17b}. For $\alpha=1$ a similar parabolic ABP estimate is obtained in \cite[Th.~3.2]{BL17b} for
$p>\frac{d}{2\underline{\mu}}+1$, and now it is seen that the second term equal to 1 is a contribution due to the
usual time-derivative.
\end{remark}

The above results are useful in the study of the inverse source problem discussed in \cite{LRY}. This problem can be
roughly stated as follows: Given $x_0\in\cD$, $T>0$ fixed, and $\varphi_0=0$, if the inhomogeneity is given in the form
$F(t, x)=\rho_1(t)\rho_2(x)$, is it possible to determine $\rho_1$ by single point observation data $\varphi(t, x_0)$
for $t\in[0, T]$?
%This problem relates with the determination of the temporal component $\rho_1$ in the inhomogeneous term $F (t, x) = \rho_1(t)\rho_2(x)$.
The spatial component $\rho_2$ simulates, for instance, a source of contaminants which may be hazardous, see
\cite{JLLY, LRY, SY11} for further discussion.

\begin{theorem}
\label{inver}
Let $\varphi_0=0$ and $F(t, x)=\rho_1(s)\rho_2(x)$. We assume that either $\rho_2\in L^\infty(\cD)$ or, in case
Assumption~\ref{WLSC} holds, $\rho_2\in L^p(\cD)$ for some $p>\frac{d}{2\underline{\mu}}$. Then for $\varphi$
satisfying \eqref{ET2.2A} we have that the map
$$
\cD \ni x\mapsto \varphi(\cdot, x)\in L^1(0, T)
$$
is continuous. Moreover, if $\rho_2\gneq 0$ and for some $x_0\in\cD$ we have $\varphi(\cdot, x_0)=0$ in $L^1(0, T)$, then $\rho_1=0$.
\end{theorem}

\begin{proof}
Recall from \eqref{ET2.3A}
\begin{equation*}
T^{\cD, V}_t f(x) =\int_{\cD} T^{\cD, V}(t, x, y) f(y)\, \D{y},\quad t>0,
\end{equation*}
where the kernel $T(t, x, y)$ is bounded for every $t>0$ and continuous in the variables $x, y\in\cD$. This implies for
every $f\in L^1(\cD)$ that the map $x\mapsto T^{\cD, V}_t f(x)$ is continuous in $\cD$, for every $t>0$. For $x\in\cD$ and
$t>0$ define
$$
\chi_t(x)= \int_0^\infty  l^{-\nicefrac{1}{\alpha}} g_1(l^{-\nicefrac{1}{\alpha}}t) T_l^{\cD, V}\rho_2 (x)\; \D{l}.
$$
If $\rho_2\in L^\infty(\cD)$, then we have $\abs{T_l^{\cD, V}\rho_2 (x)}\leq \norm{\rho_2}_{L^\infty(\cD)}$, as $V\geq 0$,
 and therefore,
by Remark~\ref{R3.1} and the dominated convergence theorem it follows that $\chi_t$ is continuous in $\cD$ for $t>0$. Also,
observe that in this situation
$$
\abs{\chi_t(x)}\leq \kappa_3 \norm{\rho_2}_{L^\infty(\cD)} (t^{\alpha-1}+1), \quad  t>0, \; x\in\cD,
$$
with a constant $\kappa_3$; see also Remark~\ref{R3.1}. Now suppose that Assumption~\ref{WLSC} holds and $\rho_2\in L^p(\cD)$
for some $p>\frac{d}{2\underline{\mu}}$. Then from the proof of Theorem~\ref{T2.4} we see that for every $x\in\cD$
$$
\abs{T_l^{\cD, V}\rho_2 (x)}\leq C\left[l^{-\frac{d}{2p\underline{\mu}}} \Ind_{(0, \kappa_2]}(l) +
\Ind_{(\kappa_2, \infty)}\right]\norm{\rho_2}_{L^p(\cD)},
$$
with a constant $C$. Thus by \eqref{Add3} it is seen that $\chi_t$ is continuous in $\cD$ for every $t>0$. Again by the proof of
Theorem~\ref{T2.4} we find
$$
\abs{\chi_t(x)}\leq \kappa_3 \norm{\rho_2}_{L^p(\cD)} (t^{\alpha-1-\frac{d\alpha}{2p\underline{\mu}}}+1), \quad t>0, \; x\in\cD,
$$
with a constant $\kappa_3$. Observe from \eqref{ET2.2A} that
$$
\varphi(t, x)=\int_0^t \rho_1(t-s)\chi_s(x) \D s.
$$
Let $x_n\to x\in\cD$. Then
\begin{align*}
\int_0^T \abs{\varphi(t, x_n)-\varphi(t, x)} \D{t}
\leq \int_0^T \int_0^t \abs{\rho_1(t-s)}\abs{\chi_s(x_n)-\chi_s(x)}\,\D{s}\,\D{t}.
\end{align*}
Using the above bounds on $\chi_t$ and its continuity in $\cD$, it follows by Young's inequality and dominated convergence that
the right hand side above tends to zero as $n\to\infty$. This shows the first part of the theorem.

To obtain the second part, note that $\chi_\cdot(x_0)\in L^1(0, T)$. By the given condition we also have
$$
\int_0^t \rho_1(s) \chi_{t-s}(x_0) \D s=0, \quad  t\in[0, T].
$$
The above equality makes sense due to the continuity result we proved above. By Titchmarsh's theorem \cite{D98, T26} there exist
non-negative $\kappa_1, \kappa_2$ with $\kappa_1+\kappa_2\geq T$ and $\rho_1=0$ almost everywhere in $(0, \kappa_1)$ and
$\chi_\cdot(x_0)=0$ almost everywhere in $(0, \kappa_2)$. However, by our assumption on $\rho_2$ it follows that $T^{\cD, V}_l\rho_2>0$
on $\cD$ (see the proof of Theorem~\ref{T2.3}) and $g_1>0$ on $(0, \infty)$, implying $\chi_\cdot(x_0)>0$ in every $(0, \kappa_2)$ for
$\kappa_2>0$. Thus $\kappa_1\geq T$, which completes the proof of the theorem.
\end{proof}

\subsection*{Acknowledgments}
This research of AB was supported in part by an INSPIRE faculty fellowship and a DST-SERB grant EMR/2016/004810.

\end{document}